\documentclass[12pt,reqno]{amsart}
\usepackage[margin=1in]{geometry}
\usepackage{amsmath,amssymb,amsthm}
\usepackage{booktabs}
\usepackage[colorlinks=true,linkcolor=blue,citecolor=blue,urlcolor=blue]{hyperref}
\newtheorem{theorem}{Theorem}[section]
\newtheorem{lemma}[theorem]{Lemma}
\newtheorem{proposition}[theorem]{Proposition}
\newtheorem{corollary}[theorem]{Corollary}
\newtheorem{conjecture}[theorem]{Conjecture}
\theoremstyle{remark}
\newtheorem{example}[theorem]{Example}

\newcommand{\N}{\mathbb{N}}
\newcommand{\Z}{\mathbb{Z}}
\newcommand{\qbinom}[2]{\genfrac{[}{]}{0pt}{}{#1}{#2}}
\newcommand{\dom}{\trianglerighteq}
\DeclareMathOperator{\sgn}{sgn}
\author{Lilan Dai}
\address{Center for Combinatorics, LPMC, Nankai University, Tianjin 300071, P. R. China}
\email{daililan@mail.nankai.edu.cn}
\thanks{}

\author{Dun Qiu}
\address{Center for Combinatorics, LPMC, Nankai University, Tianjin 300071, P. R. China}
\email{qiudun@nankai.edu.cn}
\thanks{Dun Qiu is supported in part by the National Natural Science Foundation of China (No. 12671390),
the Natural Science Foundation of Tianjin (No. 24JCZDJC01390), and the Fundamental Research Funds for the Central Universities (No. 105-63263093).}

\author{Weichen Wang}
\address{Center for Combinatorics, LPMC, Nankai University, Tianjin 300071, P. R. China}
\email{wangweichenc@163.com}
\thanks{}

\author{Minhao Zhang}
\address{Center for Combinatorics, LPMC, Nankai University, Tianjin 300071, P. R. China}
\email{carpentercheungmath@gmail.com}
\thanks{}

\title[The Schur positivity of Lucas binomial differences]{The Schur positivity of Lucas binomial differences}
\date{\today}

\begin{document}
\raggedbottom

\begin{abstract}
For positive integers $a\le b\le c\le d$ with $ad=bc$, we prove that the Lucas binomial coefficients satisfy $(-1)^{a+1}(L_{b,c}[X]-L_{a,d}[X])\in\N[e_1[X],e_2[X],h_2[X]]$, resolving Bergeron's conjecture on Schur positivity in two variables. We also prove the same positivity for $L_{u,v}[X]-e_1[X]^{uv-rs}L_{r,s}[X]$ when $rs<uv$. Using Lucas atoms, we establish $e$-positivity for degree-adjusted Lucas multinomial differences under dominance order.
\end{abstract}

\maketitle

\section{Introduction}\label{sec:intro}

The Gaussian binomial coefficients, studied by Gauss~\cite{Gau}, are $q$-analogues of ordinary binomial coefficients. O'Hara \cite{OH} gave a constructive proof of their unimodality, and Zeilberger \cite{Zei} gave an algebraic version of her proof. For positive integers $a,b,c,d$ with $ad=bc$ and $a$ the smallest parameter, Zanello \cite{Zan} proved, in response to a question of Bergeron, that $\qbinom{b+c}{b}_q-\qbinom{a+d}{a}_q$ has nonnegative and unimodal coefficients in the cases $a=2$ and $a=3$. The case $a=1$ follows from Gaussian unimodality.

Sagan and Savage \cite[Theorem~3]{SS} studied analogous quotients of \emph{Lucas polynomials}, which satisfy a two-parameter recurrence. They gave combinatorial interpretations of these quotients and established positivity in the recurrence parameters. Sagan and Tirrell \cite{ST} factored the Lucas polynomials into \emph{Lucas atoms}, providing an algebraic method for studying such quotients.

We work with the Lucas recurrence in the \emph{ring of symmetric functions}. For an \emph{alphabet} $X=x_1+x_2+\cdots$, that is, a formal sum of commuting indeterminates, the Lucas polynomials $U_n[X]$ are defined by
\begin{equation}\label{eq:U}
U_0[X]:=0,\qquad U_1[X]:=1,\qquad U_{n+1}[X]:=e_1[X]U_n[X]+e_2[X]U_{n-1}[X]\quad(n\ge1).
\end{equation}
Here $e_i[X]$ and $h_i[X]$ are the \emph{elementary} and \emph{complete homogeneous symmetric functions} in $X$, respectively. Define $U_n![X]:=\prod_{i=1}^nU_i[X]$, with $U_0![X]:=1$. For $r,s\ge0$, define the \emph{Lucas binomial coefficient} as a rational function in $e_1[X]$ and $e_2[X]$ by
\begin{equation}\label{eq:Lrs}
L_{r,s}[X]:=\frac{U_{r+s}![X]}{U_r![X]U_s![X]}=\prod_{i=1}^r\frac{U_{s+i}[X]}{U_i[X]}.
\end{equation}
The quotient is a polynomial in $e_1[X]$ and $e_2[X]$, as recalled in Subsection~\ref{sec:sym-lucas}, so $L_{r,s}[A]$ is defined for every alphabet $A$ by evaluating this polynomial at $A$. In particular, $L_{r,s}[q+t]$ is its plethystic specialization at the two-letter alphabet $q+t$. Bergeron proposed the following conjecture on \emph{Schur positivity}.

\begin{conjecture}[{\cite[Conjecture~10.1]{Ber}}]\label{conj:bergeron}
Let $a,b,c,d$ be positive integers with $a\le b$, $c\le d$, and $ad=bc$. Then $L_{b,c}[q+t]-L_{a,d}[q+t]$ is Schur positive up to a global sign.
\end{conjecture}

Under the coordinate substitution $e_1[X]\mapsto1+q$, $e_2[X]\mapsto-q$, the polynomial $L_{r,s}[X]$ becomes $\qbinom{r+s}{r}_q$ (see Subsection~\ref{sec:sym-lucas}), so Conjecture~\ref{conj:bergeron} is a Lucas analogue of the Gaussian question above. Since this substitution sends $e_2[X]$ to $-q$, our positivity results do not transfer to the Gaussian difference.

Let $\N=\{0,1,2,\ldots\}$, and write $\N[e_1[X],e_2[X],h_2[X]]$ for the set of polynomials in $e_1[X]$, $e_2[X]$, and $h_2[X]$ with coefficients in $\N$; similarly for $\N[e_1[X],e_2[X]]$. By the Pieri rule, every element of $\N[e_1[X],e_2[X],h_2[X]]$ is Schur positive. Thus Theorem~\ref{thm:equal} below, together with Corollary~\ref{cor:bergeron}, gives a stronger form of Bergeron's conjecture in the ring of symmetric functions.

\begin{theorem}\label{thm:equal}
Let $a\le b\le c\le d$ be positive integers with $ad=bc$. If $a=b$, then $L_{b,c}[X]=L_{a,d}[X]$. If $a<b$, write $ad=2M+\delta$ with $\delta\in\{0,1\}$. Then, for every alphabet $X$,
\begin{equation}\label{eq:equal}
(-1)^{a+1}(L_{b,c}[X]-L_{a,d}[X])=\sum_{j=a+1}^MC_j\,e_1[X]^\delta h_2[X]^{M-j}e_2[X]^j
\end{equation}
with integers $C_j$, independent of $X$, such that $C_{a+1}=1$ and $C_j\ge\binom{M-a-1}{j-a-1}$ for $a+1\le j\le M$. In particular, the signed difference is Schur positive as a symmetric function.
\end{theorem}

Using $L_{r,s}[X]=L_{s,r}[X]$, Theorem~\ref{thm:equal} gives the sign in Conjecture~\ref{conj:bergeron} and the cases in which the difference vanishes.

\begin{corollary}\label{cor:bergeron}
Let $a\le b$ and $c\le d$ be positive integers with $ad=bc$. Then
\begin{enumerate}
\item[(a)] $L_{b,c}[X]=L_{a,d}[X]$ for every alphabet $X$ if and only if $a=b$ or $a=c$.
\item[(b)] If $a\ne b$ and $a\ne c$, let $\varepsilon:=\sgn(c-a)(-1)^{\min(a,c)+1}$. Then $\varepsilon(L_{b,c}[X]-L_{a,d}[X])\in\N[e_1[X],e_2[X],h_2[X]]$ for every alphabet $X$, and $\varepsilon$ is the unique sign for which $\varepsilon(L_{b,c}[X]-L_{a,d}[X])$ is Schur positive for every alphabet $X$ with at least two letters.
\end{enumerate}
In particular, Conjecture~\ref{conj:bergeron} holds.
\end{corollary}

When $rs<uv$, multiplying $L_{r,s}[X]$ by $e_1[X]^{uv-rs}$ gives a polynomial of the same degree as $L_{u,v}[X]$. We also obtain a positivity result for arbitrary positive integers $r,s,u,v$ with $rs<uv$.

\begin{theorem}\label{thm:unequal}
Let $r,s,u,v$ be positive integers with $rs<uv$. Write $uv=2M+\delta$ with $\delta\in\{0,1\}$. For every alphabet $X$, we have
\[
L_{u,v}[X]-e_1[X]^{uv-rs}L_{r,s}[X]=\sum_{j=1}^MD_j\,e_1[X]^\delta h_2[X]^{M-j}e_2[X]^j
\]
with integers $D_j$, independent of $X$, such that $D_1=uv-rs$ and $D_j\ge(uv-rs)\binom{M-1}{j-1}$ for $1\le j\le M$. The inequalities for the $D_j$ are equivalent to
\begin{equation}\label{eq:unequal-base}
L_{u,v}[X]-e_1[X]^{uv-rs}L_{r,s}[X]-(uv-rs)e_1[X]^{uv-2}e_2[X]\in\N[e_1[X],e_2[X],h_2[X]].
\end{equation}
\end{theorem}

For a partition $\alpha=(\alpha_1,\ldots,\alpha_\ell)$ of $n\ge1$, we consider the symmetric-function version of the \emph{Lucas multinomial coefficient} \cite{Cera}
\[
L_\alpha[X]:=\frac{U_n![X]}{U_{\alpha_1}![X]\cdots U_{\alpha_\ell}![X]},
\]
where the quotient is taken in the fraction field of $\Z[e_1[X],e_2[X]]$. In particular, $L_{(s,r)}[X]=L_{r,s}[X]$. Repeated use of \eqref{eq:Lrs} shows that $L_\alpha[X]$ is a polynomial in $e_1[X]$ and $e_2[X]$, so $L_\alpha[A]$ is again defined for every alphabet $A$. Let $d(\alpha):=\deg L_\alpha[X]$ in the grading $\deg e_1[X]=1$, $\deg e_2[X]=2$; by Lemma~\ref{lem:factor}, $d(\alpha)=\sum_{i<j}\alpha_i\alpha_j$. We write $\ell(\alpha)$ for the number of parts of $\alpha$. We say that $\alpha$ dominates $\beta$, written $\alpha\unrhd\beta$, if $\alpha_1+\cdots+\alpha_k\ge\beta_1+\cdots+\beta_k$ for all $k$, where missing parts are taken to be zero.

\begin{theorem}\label{thm:multi}
Let $\alpha\dom\beta$ be partitions of $n\ge1$, and let $\kappa:=d(\beta)-d(\alpha)-\ell(\beta)+\ell(\alpha)$. Then $d(\beta)-d(\alpha)\ge0$, $\kappa\ge0$, and, for every alphabet $X$,
\begin{equation}\label{eq:multi-positive}
L_\beta[X]-e_1[X]^{d(\beta)-d(\alpha)}L_\alpha[X]\in\N[e_1[X],e_2[X]].
\end{equation}
If $\kappa>0$, then $d(\beta)\ge2$ and
\begin{equation}\label{eq:multi-base}
L_\beta[X]-e_1[X]^{d(\beta)-d(\alpha)}L_\alpha[X]-\kappa\,e_1[X]^{d(\beta)-2}e_2[X]\in\N[e_1[X],e_2[X]].
\end{equation}
Moreover, $L_\beta[X]=e_1[X]^{d(\beta)-d(\alpha)}L_\alpha[X]$ for every alphabet $X$ if and only if $\kappa=0$, and this happens exactly when $\beta$ is obtained from $\alpha$ by replacing some parts equal to $2$ by pairs of parts equal to $1$ (possibly none).
\end{theorem}

Unlike the multinomial differences in Theorem \ref{thm:multi}, the signed binomial differences in Theorem \ref{thm:equal} need not be $e$-positive. For example, for $(a,b,c,d)=(2,4,4,8)$, the coefficient of $e_2[X]^8$ in the signed difference $L_{2,8}[X]-L_{4,4}[X]$ is $-1$.

Section~\ref{sec:background} recalls the symmetric function, Lucas specializations, and coefficient estimates used below. Section~\ref{sec:equal} proves Theorem~\ref{thm:equal}, Corollary~\ref{cor:bergeron}, and Theorem~\ref{thm:unequal} by extracting coefficients after the substitution $e_1[X]\mapsto1$, $e_2[X]\mapsto z/(1-z)^2$, which turns Lucas binomial coefficients into Gaussian polynomials at $q=-z$, and by bounding finite products. Section~\ref{sec:multi} proves Theorem~\ref{thm:multi} by comparing Lucas atoms of consecutive partitions in a dominance chain. 

A Lean~4 formalization of Theorem~\ref{thm:equal}, Theorem ~\ref{thm:unequal} and Theorem~\ref{thm:multi} is available at \url{https://github.com/Lilan-Dai/lucas-schur-positivity}.

\section{Preliminaries}\label{sec:background}

In Subsection~\ref{sec:sym-lucas} we recall the symmetric-function identities and Lucas specializations; see Macdonald \cite{Mac} and Sagan and Tirrell \cite{ST}. In Subsection~\ref{sec:coords} we derive the coefficient formulas used for the binomial differences. In Subsection~\ref{sec:criterion} we establish the coefficient estimates, using the factorization of words into Lyndon words.

\subsection{Symmetric functions and Lucas polynomials}\label{sec:sym-lucas}

A \emph{partition} is a weakly decreasing sequence $\lambda=(\lambda_1,\lambda_2,\ldots)$ of nonnegative integers with finitely many nonzero terms; trailing zeros are ignored. Its \emph{length} is $\ell(\lambda):=\#\{i:\lambda_i>0\}$ and its \emph{size} (or \emph{weight}) is $|\lambda|:=\sum_i\lambda_i$. We write $\lambda\vdash n$ if $|\lambda|=n$.

Let $X=x_1+x_2+\cdots$ be an alphabet of indeterminates, and let $\Lambda$ be the ring of symmetric functions in $X$ over $\Z$. We write $e_r[X]$, $h_r[X]$, and $s_\lambda[X]$ for the elementary, complete homogeneous, and Schur symmetric functions in $X$. In particular,
\[
\sum_{r\ge0}e_r[X]z^r=\prod_{i\ge1}(1+x_i z),\qquad
\sum_{r\ge0}h_r[X]z^r=\prod_{i\ge1}(1-x_i z)^{-1}.
\]
We use $f[A]$ for \emph{plethystic substitution}; Macdonald denotes plethysm by $f\circ g$ \cite{Mac}. Writing $p_m[X]:=\sum_i x_i^m$ for $m\ge1$, substitution at an alphabet $A=a_1+a_2+\cdots$ is the ring homomorphism determined over $\mathbb{Q}$ by $p_m[A]=\sum_i a_i^m$. Thus $p_m[q+t]=q^m+t^m$ and $f[q+t]=f(q,t,0,\ldots)$. The elements $e_1[X]$ and $e_2[X]$ are algebraically independent in $\Lambda$.

The recurrence \eqref{eq:U} gives $U_2[X]=e_1[X]$, $U_3[X]=e_1[X]^2+e_2[X]$, and $U_4[X]=e_1[X]^3+2e_1[X]e_2[X]$. More generally,
\begin{equation}\label{eq:U-explicit}
U_n[X]=\sum_{k=0}^{\lfloor(n-1)/2\rfloor}\binom{n-1-k}{k}e_1[X]^{n-1-2k}e_2[X]^k\qquad(n\ge1).
\end{equation}
Clearly $L_{r,s}[X]=L_{s,r}[X]$. For $F[X]\in\Z[e_1[X],e_2[X]]$, write $F(u,v)$ for the \emph{coordinate substitution} $e_1[X]\mapsto u$, $e_2[X]\mapsto v$. Thus $F(1,y)$ is obtained by substituting $1$ and $y$ for the two generators, whereas evaluation at the two-letter alphabet $q+t$ gives $e_1[q+t]=q+t$, $e_2[q+t]=qt$, and $h_2[q+t]=q^2+qt+t^2$. 

Under the coordinate substitution $(e_1[X],e_2[X])=(1,1)$, the polynomials $L_{r,s}[X]$ become the \emph{Fibonomial coefficients}; at $(e_1[X],e_2[X])=(2,-1)$, the recurrence gives $U_n(2,-1)=n$ and $L_{r,s}(2,-1)=\binom{r+s}{r}$. Sagan and Tirrell \cite{ST} observed that $(e_1[X],e_2[X])=(1+q,-q)$ gives $U_n(1+q,-q)=[n]_q$ (in the notation recalled below), hence $L_{r,s}(1+q,-q)=\qbinom{r+s}{r}_q$. This is a coordinate substitution, distinct from evaluation at $q+t$.
Sagan and Savage \cite{SS} proved that $L_{r,s}[X]\in\N[e_1[X],e_2[X]]$; see also Sagan and Tirrell \cite{ST}. Its degree is $rs$ when $\deg e_1[X]=1$ and $\deg e_2[X]=2$.

We write $[y^k]f(y)$ for the coefficient of $y^k$ in $f(y)$, with $[y^k]f(y)=0$ for $k<0$. We also set $\binom nk=0$ for $k<0$. For $F[X]\in\Z[e_1[X],e_2[X]]$, $[e_1[X]^ie_2[X]^k]F[X]$ denotes its coefficient in the polynomial basis in $e_1[X]$ and $e_2[X]$. For a \emph{formal power series} $f(z)$, the notation $f(z)=O(z^k)$ means that $z^k$ divides $f(z)$; we use the same convention with other variables.

We use the identities $e_1[X]=s_{(1)}[X]$,
$e_2[X]=s_{(1,1)}[X]$, and $h_2[X]=s_{(2)}[X]$.
The \emph{Pieri rule} \cite{Mac}
states that, for any partition $\mu$ and integer $r\ge0$,
\[
s_\mu[X]h_r[X]
=\sum_{\lambda/\mu\text{ horizontal }r\text{-strip}}s_\lambda[X],
\qquad
s_\mu[X]e_r[X]
=\sum_{\lambda/\mu\text{ vertical }r\text{-strip}}s_\lambda[X].
\]
Here a skew diagram $\lambda/\mu$ is a
\emph{horizontal $r$-strip} if it has $r$ cells and at most
one cell in each column, and a \emph{vertical $r$-strip}
if it has $r$ cells and at most one cell in each row.

A symmetric function is \emph{Schur positive} if its Schur coefficients are nonnegative integers. For $X=x_1+\cdots+x_k$, the polynomials $s_\lambda[X]$ with $\ell(\lambda)\le k$ are linearly independent, while $s_\lambda[X]=0$ for $\ell(\lambda)>k$; a symmetric polynomial in $x_1,\ldots,x_k$ is \emph{Schur positive} if it is a nonnegative integer combination of these $s_\lambda[X]$. Thus Schur positivity is preserved under specialization to a finite alphabet. For two letters and $0\le j\le\lfloor N/2\rfloor$, we have
\begin{equation}\label{eq:two-row}
s_{(N-j,j)}[q+t]=(qt)^j\,[N-2j+1]_{q,t},
\end{equation}
where $[k]_{q,t}:=q^{k-1}+q^{k-2}t+\cdots+t^{k-1}$ for $k\ge1$.

Give $\Z[e_1[X],e_2[X]]$ the grading $\deg e_1[X]=1$ and $\deg e_2[X]=2$. Since $q+t$ and $qt$ are algebraically independent, the map $F[X]\mapsto F[q+t]$ is injective on $\Z[e_1[X],e_2[X]]$. Consequently, identities in these two generators may be checked in two variables. An element $F[X]\in\Z[e_1[X],e_2[X]]$ is \emph{$e$-positive} if $F[X]\in\N[e_1[X],e_2[X]]$. Also, $h_2[X]=e_1[X]^2-e_2[X]$, and the Pieri rule shows that products of $e_1[X],e_2[X],h_2[X]$ are Schur positive.

Write $N=2M+\delta$, where $\delta\in\{0,1\}$. The monomials $e_1[X]^{N-2k}e_2[X]^k$, $0\le k\le M$, form a basis of the homogeneous polynomials of degree $N$ in $e_1[X]$ and $e_2[X]$. Hence the coordinate substitution $(e_1[X],e_2[X])=(1,y)$ identifies these homogeneous polynomials with the polynomials in $y$ of degree at most $M$:
\begin{equation}\label{eq:dehom}
F[X]=\sum_{k=0}^M\gamma_k\,e_1[X]^{N-2k}e_2[X]^k\qquad\Longleftrightarrow\qquad F(1,y)=\sum_{k=0}^M\gamma_ky^k.
\end{equation}
The inverse replaces $y^k$ by $e_1[X]^{N-2k}e_2[X]^k$. For example, \eqref{eq:U-explicit} gives
\[
U_n(1,y)=\sum_{k=0}^{\lfloor(n-1)/2\rfloor}\binom{n-1-k}{k}y^k.
\]
This is the $n$th Fibonacci polynomial; its value $F_n:=U_n(1,1)$ is the $n$th Fibonacci number.

Since $U_1(1,y)=1$ and $U_n(1,y)=1+(n-2)y+\cdots$ for $n\ge2$, the product formula \eqref{eq:Lrs} gives, for $r,s\ge1$,
\begin{equation}\label{eq:Lrs-lin}
L_{r,s}(1,0)=1,\qquad [y]\,L_{r,s}(1,y)=\sum_{i=1}^r(s+i-2)-\sum_{i=2}^r(i-2)=rs-1.
\end{equation}

For Gaussian polynomials we use the standard notation $(q;q)_n:=(1-q)(1-q^2)\cdots(1-q^n)$, $[n]_q:=(1-q^n)/(1-q)$ and
\[
\qbinom{n}{k}_q:=\frac{(q;q)_n}{(q;q)_k\,(q;q)_{n-k}}.
\]
We also use the homogenized Gaussian polynomial $\qbinom{n}{k}_{q,t}:=t^{k(n-k)}\qbinom{n}{k}_{q/t}$. Let $\tau$ be the ring automorphism of $\Z[e_1[X],e_2[X]]$ with $\tau(e_1[X])=e_1[X]$ and $\tau(e_2[X])=-e_2[X]$. We use the same notation for the automorphism of the ring of symmetric polynomials in $q,t$ that corresponds to $\tau$ under the isomorphism $F[X]\mapsto F[q+t]$. Then $\tau(U_{n+1}[X])=e_1[X]\tau(U_n[X])-e_2[X]\tau(U_{n-1}[X])$, which becomes the recurrence for $h_{n-1}[q+t]=[n]_{q,t}$ under $e_1[q+t]=q+t$ and $e_2[q+t]=qt$. The initial values agree, so
\begin{equation}\label{eq:tau}
\tau(U_n[q+t])=[n]_{q,t},\qquad
\tau(L_{r,s}[q+t])=\qbinom{r+s}{r}_{q,t}
=t^{rs}\qbinom{r+s}{r}_{q/t}.
\end{equation}
\subsection{Coefficient extraction}\label{sec:coords}

\begin{lemma}\label{lem:basis}
Let $F[X]\in\Z[e_1[X],e_2[X]]$ be homogeneous of degree $N=2M+\delta$, where $\delta\in\{0,1\}$. Define integers $c_j(F)$ by
\[
F(1,y)=\sum_{j=0}^Mc_j(F)\,y^j(1-y)^{M-j}.
\]
\begin{enumerate}
\item[(a)] $F[X]=\sum_{j=0}^Mc_j(F)\,e_1[X]^\delta h_2[X]^{M-j}e_2[X]^j$. In particular, these products form a basis of the homogeneous polynomials of degree $N$ in $e_1[X]$ and $e_2[X]$.
\item[(b)] $F[X]\in\N[e_1[X],e_2[X],h_2[X]]$ if and only if $c_j(F)\ge0$ for all $j$.
\item[(c)] If $F(1,y)=\sum_{k=0}^M\gamma_ky^k$, the coefficients $c_j(F)$ are given equivalently by
\begin{equation}\label{eq:transform}
\begin{gathered}
c_j(F)=\sum_{k=0}^j\binom{M-k}{j-k}\gamma_k\quad(0\le j\le M)
\textup{ and }\sum_{j=0}^Mc_j(F)\,w^j=(1+w)^MF\Bigl(1,\frac{w}{1+w}\Bigr).
\end{gathered}
\end{equation}
\end{enumerate}
\end{lemma}

\begin{proof}
Since $h_2[X]=e_1[X]^2-e_2[X]$, the coordinate substitution $(e_1[X],e_2[X])=(1,y)$ sends the product $e_1[X]^\delta h_2[X]^{M-j}e_2[X]^j$ to $y^j(1-y)^{M-j}$. These polynomials have lowest terms $y^j$, so they form a $\Z$-basis of the polynomials of degree at most $M$. Part~(a) follows from \eqref{eq:dehom}. For part~(c), the binomial theorem gives
\[
y^k=y^k\bigl(y+(1-y)\bigr)^{M-k}=\sum_{j=k}^M\binom{M-k}{j-k}y^j(1-y)^{M-j}.
\]
Substituting $y=w/(1+w)$ gives $y^j(1-y)^{M-j}=w^j(1+w)^{-M}$, which proves the generating function. For part~(b), if every $c_j(F)\ge0$, part~(a) gives $F[X]\in\N[e_1[X],e_2[X],h_2[X]]$. Conversely, suppose that $F[X]\in\N[e_1[X],e_2[X],h_2[X]]$. Taking homogeneous components of degree $N$, we may write $F[X]$ as a nonnegative integer combination of monomials $e_1[X]^{\delta+2p}e_2[X]^vh_2[X]^u$ of degree $N$, so that $p+u+v=M$. The coordinate substitution of such a monomial is
\[
y^v(1-y)^u=y^v(1-y)^u\bigl(y+(1-y)\bigr)^p,
\]
which has nonnegative coefficients in the basis $y^j(1-y)^{M-j}$. Since each $c_j$ is linear, every $c_j(F)\ge0$.
\end{proof}

By Lemma~\ref{lem:basis}(a), $c_j(F)$ is the coefficient of $e_1[X]^\delta h_2[X]^{M-j}e_2[X]^j$ in $F[X]$.

Although \eqref{eq:transform} expresses the coefficients $c_j(F)$ in terms
of $F(1,y)$, working directly with $L_{r,s}(1,y)$ is inconvenient. The following substitution $y=z/(1-z)^2$ relates Lucas binomial coefficients to Gaussian binomial
coefficients at $q=-z$, whose product formula will be useful for the coefficient estimates below.

\begin{lemma}\label{lem:binet}
For $n\ge1$,
\begin{equation}\label{eq:binet}
U_n\Bigl(1,\frac{z}{(1-z)^2}\Bigr)=\frac{[n]_{-z}}{(1-z)^{n-1}},
\end{equation}
where $[n]_{-z}=(1-(-z)^n)/(1+z)$.
Consequently, for $r,s\ge0$,
\begin{equation}\label{eq:binet-L}
L_{r,s}\Bigl(1,\frac{z}{(1-z)^2}\Bigr)=\frac{1}{(1-z)^{rs}}\qbinom{r+s}{r}_{-z}.
\end{equation}
\end{lemma}

\begin{proof}
Let $\xi_1=1/(1-z)$ and $\xi_2=-z/(1-z)$. Then
\[
\xi_1+\xi_2=1,\qquad \xi_1\xi_2=-\frac{z}{(1-z)^2}.
\]
The sequence $u_n:=(\xi_1^n-\xi_2^n)/(\xi_1-\xi_2)$ satisfies $u_0=0$, $u_1=1$ and $u_{n+1}=(\xi_1+\xi_2)u_n-\xi_1\xi_2u_{n-1}$. These are the recurrence and initial conditions of $U_n(1,z/(1-z)^2)$. Since $\xi_2=-z\xi_1$, we obtain
\[
U_n\Bigl(1,\frac{z}{(1-z)^2}\Bigr)=\frac{\xi_1^n-\xi_2^n}{\xi_1-\xi_2}=\xi_1^{n-1}\,\frac{1-(-z)^n}{1+z}=\frac{[n]_{-z}}{(1-z)^{n-1}}.
\]
Applying \eqref{eq:binet} to the factors in \eqref{eq:Lrs} gives \eqref{eq:binet-L}, since the exponent of $1-z$ in the denominator is
\[
\sum_{i=1}^r\bigl((s+i-1)-(i-1)\bigr)=rs.\qedhere
\]
\end{proof}

Lemma~\ref{lem:binet} is the case $(e_1[X],e_2[X])=(1,z/(1-z)^2)$ of the classical identity $U_n[X]=h_{n-1}[\theta_1+\theta_2]$, where $\theta_1$ and $\theta_2$ are the roots of $T^2-e_1[X]\,T-e_2[X]$; in the proof above, $\xi_1$ and $\xi_2$ are these roots.

To apply Lemma~\ref{lem:binet} to \eqref{eq:transform}, let $y=z/(1-z)^2$. Then $1-y=(1-3z+z^2)/(1-z)^2$, so $w=y/(1-y)$ satisfies
\begin{equation}\label{eq:w}
w=\frac{z}{1-3z+z^2},\qquad
1+w=\frac{(1-z)^2}{1-3z+z^2},\qquad
\frac{w}{1+w}=y.
\end{equation}
We use the change of variables in \eqref{eq:w} together with
\[
\frac{1}{1-3z+z^2}=\sum_{n\ge0}F_{2n+2}\,z^n
=1+3z+8z^2+21z^3+55z^4+\cdots,
\]
whose coefficients are Fibonacci numbers.

\begin{lemma}\label{lem:lagrange}
Let $z=z(w)$ be the compositional inverse of $w=z/(1-3z+z^2)$. For every power series $\Phi(z)$ and every $n\ge0$,
\begin{equation}\label{eq:lagrange}
[w^n]\,\Phi(z(w))
=[z^n]\,\Phi(z)(1-3z+z^2)^{n-1}(1-z^2).
\end{equation}
\end{lemma}

\begin{proof}
Since $1-3z+z^2$ has constant term $1$, the inverse $z(w)$ exists and satisfies $w=z(w)/(1-3z(w)+z(w)^2)$. Stanley \cite[Theorem~5.4.2]{EC2} gives, in the present notation,
\[
n[w^n]z(w)^k=k[z^{n-k}](1-3z+z^2)^n\qquad(n,k\ge1).
\]
Both sides of \eqref{eq:lagrange} are linear in $\Phi(z)$, and each coefficient depends on only finitely many terms. It therefore suffices to take $\Phi(z)=z^k$. Since $(1-3z+z^2)-z(-3+2z)=1-z^2$, for $n,k\ge1$ we have
\begin{align*}
[z^n]z^k(1-3z+z^2)^{n-1}(1-z^2)
&=[z^{n-k}](1-3z+z^2)^n
-[z^{n-k-1}](1-3z+z^2)^{n-1}(-3+2z)\\
&=[z^{n-k}](1-3z+z^2)^n
-\frac1n[z^{n-k-1}]\frac{d}{dz}(1-3z+z^2)^n\\
&=\frac kn[z^{n-k}](1-3z+z^2)^n=[w^n]z(w)^k.
\end{align*}
When $k=0$ and $n\ge1$, the same calculation gives zero on both sides. When $n=0$, both sides equal the constant term of $z^k$.
\end{proof}

\begin{proposition}\label{prop:kernel}
Let $F[X]\in\Z[e_1[X],e_2[X]]$ be homogeneous of degree $N=2M+\delta$, where $\delta\in\{0,1\}$, and suppose that
\begin{equation}\label{eq:kernel-hyp}
F\Bigl(1,\frac{z}{(1-z)^2}\Bigr)=\frac{G(z)}{(1-z)^N}
\end{equation}
for a power series $G(z)$. Then, for $0\le j\le M$,
\begin{equation}\label{eq:kernel}
c_j(F)=[z^j]\,\frac{(1-z^2)G(z)}{(1-z)^\delta(1-3z+z^2)^{M-j+1}}.
\end{equation}
In particular, if $G(z)=z^mG_0(z)$ with $0\le m\le M$ and $G_0$ a power series, then $c_j(F)=0$ for $j<m$ and $c_m(F)=G_0(0)$.
\end{proposition}

\begin{proof}
By Lemma~\ref{lem:basis}(c) and \eqref{eq:w}, after the substitution $w=z/(1-3z+z^2)$ we have
\[
\sum_{j=0}^Mc_j(F)\,w^j=(1+w)^MF(1,y)=\frac{(1-z)^{2M}}{(1-3z+z^2)^M}\cdot\frac{G(z)}{(1-z)^N}=\frac{G(z)}{(1-z)^\delta\,(1-3z+z^2)^M}.
\]
Let $\Phi(z)=G(z)(1-z)^{-\delta}(1-3z+z^2)^{-M}$. Substituting $z=z(w)$ gives $\sum_jc_j(F)\,w^j=\Phi(z(w))$. Apply Lemma~\ref{lem:lagrange} with $n=j$; combining the factor $(1-3z+z^2)^{j-1}$ of \eqref{eq:lagrange} with $(1-3z+z^2)^{-M}$ gives \eqref{eq:kernel}.
\end{proof}

\subsection{Coefficient estimates}\label{sec:criterion}

For the binomial differences considered below, the series $G(z)$ in \eqref{eq:kernel-hyp} has the form
\[
G(z)=\frac{z^mJ(z)}{(-z;-z)_b},
\]
where $(-z;-z)_b=\prod_{i=1}^b\bigl(1-(-z)^i\bigr)$, $J(z)\in\Z[z]$ and $J(0)=1$. We give a sufficient condition on the negative coefficients of $J(z)$ for the coefficients in \eqref{eq:kernel} to be nonnegative. All infinite products below are interpreted as formal power series; each coefficient involves only finitely many factors.

For a power series $f(z)=\sum_nf_nz^n$ with real coefficients, we write
\[
f^+(z):=\sum_{f_n>0}f_nz^n,\qquad f^-(z):=\sum_{f_n<0}|f_n|\,z^n.
\]
Thus $f(z)=f^+(z)-f^-(z)$. Write $f(z)\succeq g(z)$ if $f(z)-g(z)$ has nonnegative coefficients. For a polynomial $J(z)=\sum_nj_nz^n$, the quantity used below is $J^-(1/2)=\sum_{j_n<0}|j_n|2^{-n}$.

To control the negative coefficients of $J(z)$, we use the series
\[
V(z):=\frac{1-z^2}{1-3z+z^2}\prod_{i\ge1,\ i\text{ odd}}(1-z^i)=\sum_{n\ge0}v_nz^n,
\]
so that $(v_0,v_1,v_2,\ldots)=(1,2,4,10,27,71,186,\ldots)$.
We prove $v_n\ge2v_{n-1}$ for $n\ge1$ using the factorization of the Lyndon word.

Consider a totally ordered alphabet whose letters have positive integer weights, with only finitely many letters of each weight. The weight of a word is the sum of its letter weights. A nonempty word is a \textit{Lyndon word} if it is smaller in lexicographic order than each of its nonempty proper suffixes. For example, when $x<y$, the words $x$, $xy$ and $xxy$ are Lyndon, but $xx$ and $yx$ are not. Every word has a unique factorization $w_1w_2\cdots w_k$ into Lyndon words satisfying $w_1\ge\cdots\ge w_k$. This is the theorem of Chen, Fox and Lyndon \cite{CFL}; see also Lothaire \cite{Lot}.

\begin{lemma}\label{lem:lyndon}
Let $u(z)=\sum_{n\ge1}u_nz^n$ have nonnegative integer coefficients. Then there are integers $\lambda_n\ge u_n$ such that
\[
\frac{1}{1-u(z)}=\prod_{n\ge1}(1-z^n)^{-\lambda_n}.
\]
\end{lemma}

\begin{proof}
Choose a totally ordered set of letters with $u_n$ letters of weight $n$. Since there are finitely many letters of each weight, there are finitely many words of each weight, so the generating functions below are well defined. The generating function of all words by weight is $(1-u(z))^{-1}$. The factorization gives a weight-preserving bijection between words and finite multisets of Lyndon words. Thus, if $\lambda_n$ is the number of Lyndon words of weight $n$, the generating function of all words is also $\prod_n(1-z^n)^{-\lambda_n}$. Each single letter is a Lyndon word, so $\lambda_n\ge u_n$.
\end{proof}

\begin{lemma}\label{lem:V}
The series $(1-2z)V(z)$ has nonnegative integer coefficients. Hence $v_n\ge2v_{n-1}$ for $n\ge1$. Since $v_0=1$, every $v_n$ is positive, and $v_{n-k}\le2^{-k}v_n$ for $0\le k\le n$.
\end{lemma}

\begin{proof}
Let $T(z):=(1-z)/(1-3z+z^2)=\sum_{n\ge0}F_{2n+1}z^n=1+2z+5z^2+13z^3+\cdots$. Since
\[
1-\frac{z^2}{(1-z)(1-2z)}=\frac{(1-z)(1-2z)-z^2}{(1-z)(1-2z)}=\frac{1-3z+z^2}{(1-z)(1-2z)},
\]
we have
\[
(1-2z)\,T(z)=\frac{1}{1-u(z)},\qquad u(z):=\frac{z^2}{(1-z)(1-2z)}=\sum_{n\ge2}(2^{n-1}-1)z^n.
\]

Since $u_1=0$, comparing coefficients of $z$ gives $\lambda_1=0$. By Lemma~\ref{lem:lyndon}, 
\[
(1-2z)\,T(z)=\prod_{n\ge2}(1-z^n)^{-\lambda_n},\qquad \lambda_n\ge2^{n-1}-1\ge1.
\]
Since $(1-z^2)/(1-3z+z^2)=(1+z)T(z)$ and $(1+z)(1-z)=1-z^2$, we obtain
\[
V(z)=T(z)\,(1-z^2)\prod_{i\ge3,\ i\text{ odd}}(1-z^i).
\]
Therefore
\[
(1-2z)\,V(z)=\prod_{n\ge2}(1-z^n)^{-\lambda_n}\cdot(1-z^2)\prod_{i\ge3,\ i\text{ odd}}(1-z^i)=\prod_{n\ge2}(1-z^n)^{-\eta_n},
\]
where $\eta_n=\lambda_n-1$ if $n=2$ or $n$ is odd, and $\eta_n=\lambda_n$ otherwise. Since every $\eta_n$ is nonnegative, this product has nonnegative integer coefficients. The lemma follows, because $(1-2z)V(z)=v_0+\sum_{n\ge1}(v_n-2v_{n-1})z^n$.
\end{proof}

The estimate $v_{n-k}\le2^{-k}v_n$ gives the following positivity result.

\begin{lemma}\label{lem:absorb}
Let $J(z)\in\Z[z]$ with $J(0)=1$ and $J^-(1/2)<1$. Then for every $b\ge1$ all coefficients of
\[
W(z):=\frac{(1-z^2)J(z)}{(1-3z+z^2)(-z;-z)_b}
\]
are positive integers. In particular, $W(z)\succeq1/(1-z)$.
\end{lemma}

\begin{proof}
Write $J(z)=\sum_kj_kz^k$. By Lemma~\ref{lem:V},
\[
[z^n]\,V(z)J(z)\ge v_n-\sum_{\substack{1\le k\le n\\ j_k<0}}|j_k|\,v_{n-k}\ge v_n\bigl(1-J^-(1/2)\bigr)>0.
\]
For odd $i$, the identity $(1+z^i)^{-1}=(1-z^i)/(1-z^{2i})$ gives
\[
\frac{1-z^2}{(1-3z+z^2)(-z;-z)_b}=V(z)\prod_{\substack{i>b\\ i\text{ odd}}}\frac{1}{1-z^i}\prod_{\substack{i\le b\\ i\text{ odd}}}\frac{1}{1-z^{2i}}\prod_{\substack{i\le b\\ i\text{ even}}}\frac{1}{1-z^i}.
\]
Each of the three products has nonnegative integer coefficients and constant term $1$. Multiplying them by $V(z)J(z)$ therefore preserves strict positivity of every coefficient.
\end{proof}

\begin{corollary}\label{cor:criterion}
Let $F[X]\in\Z[e_1[X],e_2[X]]$ be homogeneous of degree $N=2M+\delta$, where $\delta\in\{0,1\}$, and suppose that
\begin{equation}\label{eq:criterion-hyp}
F\Bigl(1,\frac{z}{(1-z)^2}\Bigr)=\frac{z^mJ(z)}{(1-z)^N(-z;-z)_b}
\end{equation}
with $0\le m\le M$, $b\ge1$, $J(z)\in\Z[z]$, $J(0)=1$ and $J^-(1/2)<1$. Then $c_j(F)=0$ for $j<m$, $c_m(F)=1$, and
\[
c_j(F)\ge\binom{M-m}{j-m}\qquad(m\le j\le M).
\]
In particular, $F[X]-e_1[X]^{N-2m}e_2[X]^m\in\N[e_1[X],e_2[X],h_2[X]]$.
\end{corollary}

\begin{proof}
By Proposition~\ref{prop:kernel} with $G(z)=z^mJ(z)/(-z;-z)_b$,
\[
c_j(F)=[z^{j-m}]\,\frac{W(z)}{(1-z)^\delta\,(1-3z+z^2)^{M-j}}.
\]
Here $W(z)$ is the series in Lemma~\ref{lem:absorb}. Its constant term is $1$, so $c_j(F)=0$ for $j<m$ and $c_m(F)=1$. Moreover, $W(z)\succeq1/(1-z)$, $1/(1-3z+z^2)\succeq1/(1-z)$ and $(1-z)^{-\delta}\succeq1$. All these series have nonnegative coefficients. Thus
\[
c_j(F)\ge[z^{j-m}]\,\frac{1}{(1-z)^{M-j+1}}=\binom{M-m}{j-m}.
\]
By \eqref{eq:transform}, the coefficient of $e_1[X]^\delta h_2[X]^{M-j}e_2[X]^j$ in $e_1[X]^{N-2m}e_2[X]^m$ is $\binom{M-m}{j-m}$. Subtracting this polynomial from $F[X]$ and applying Lemma~\ref{lem:basis}(b) proves the last assertion.
\end{proof}

To apply Corollary~\ref{cor:criterion}, we need estimates for products of factors $1-(-z)^i$. For a finite set $I$ of positive integers, define the product $\Pi_I(z)$ and the odd and even weighted sums of $I$ by
\[
\Pi_I(z):=\prod_{i\in I}\bigl(1-(-z)^i\bigr),\qquad
w_{\mathrm{odd}}(I):=\sum_{\substack{i\in I\\i\text{ odd}}}2^{-i},\qquad
w_{\mathrm{even}}(I):=\sum_{\substack{i\in I\\i\text{ even}}}2^{-i}.
\]
The factor $1-(-z)^i$ is $1+z^i$ for odd $i$ and $1-z^i$ for even $i$. Expanding over subsets of $I$ gives
\begin{equation}\label{eq:subsets}
\Pi_I(z)=\sum_{S\subseteq I}(-1)^{\#\{i\in S:\ i\text{ even}\}}\,z^{\sum_{i\in S}i}.
\end{equation}
We next bound $(\Pi_I(z))^-(1/2)$ and $(\Pi_I(z)-1)^+(1/2)$ in terms of $w_{\mathrm{odd}}(I)$ and $w_{\mathrm{even}}(I)$.
\begin{lemma}\label{lem:subset}
Let $I$ be a finite set of positive integers with $w_{\mathrm{odd}}(I)<1$ and $w_{\mathrm{even}}(I)<1$. Then
\begin{align*}
(\Pi_I(z))^-(1/2)\le\frac{w_{\mathrm{even}}(I)}{(1-w_{\mathrm{odd}}(I))(1-w_{\mathrm{even}}(I))}, (\Pi_I(z)-1)^+(1/2)\le\frac{w_{\mathrm{odd}}(I)+w_{\mathrm{even}}(I)^2}{(1-w_{\mathrm{odd}}(I))(1-w_{\mathrm{even}}(I))}.
\end{align*}
\end{lemma}

\begin{proof}
Write $x=w_{\mathrm{odd}}(I)$ and $y=w_{\mathrm{even}}(I)$. For $h,k\ge0$, let $a_h$ and $b_k$ be the sums of $2^{-\sum_{i\in S}i}$ over the $h$-element subsets of the odd indices in $I$ and the $k$-element subsets of the even indices in $I$, respectively. Then $a_0=b_0=1$, $a_h\le x^h$, and $b_k\le y^k$, since each product $2^{-\sum_{i\in S}i}$ occurs among the nonnegative terms in the expansion of $x^h$ or $y^k$.

By \eqref{eq:subsets}, a subset $S$ gives a negative term only if $S$ contains an odd number of even indices, hence at least one. Collecting terms of equal degree can only decrease the total weight of the negative terms. Therefore
\[
(\Pi_I(z))^-(1/2)
\le\sum_{h\ge0}a_h\sum_{k\ge1}b_k
\le\frac{y}{(1-x)(1-y)}.
\]
A nonempty subset giving a positive term either contains at least one odd index, or contains no odd indices and at least two even indices. These two cases are disjoint, and collecting terms can again only decrease the total weight of the positive terms. Thus
\[
(\Pi_I(z)-1)^+(1/2)\le\sum_{h\ge1}a_h\sum_{k\ge0}b_k+\sum_{k\ge2}b_k\le\frac{x}{(1-x)(1-y)}+\frac{y^2}{1-y}\le\frac{x+y^2}{(1-x)(1-y)}.\qedhere
\]
\end{proof}

\section{Lucas binomial differences}\label{sec:equal}

In Subsection~\ref{sec:equal-proof} we use the coefficient formula of Subsection~\ref{sec:coords} and the estimate of Subsection~\ref{sec:criterion} to prove Theorem~\ref{thm:equal}; Corollary~\ref{cor:bergeron} then follows by symmetry. The case $a=1$ uses Lemma~\ref{lem:KOH}. In Subsection~\ref{sec:unequal} we compare both Lucas binomial coefficients with the one-row coefficients $L_{1,n}[X]=U_{n+1}[X]$ and apply the same coefficient estimate to prove Theorem~\ref{thm:unequal}.

\subsection{Proof of Theorem~\ref{thm:equal}}\label{sec:equal-proof}

Fix $a<b\le c\le d$ with $ad=bc$, and write $ad=2M+\delta$, where $\delta\in\{0,1\}$. We use Corollary~\ref{cor:criterion} when $a\ge2$, and Lemma~\ref{lem:KOH} when $a=1$.

Let $I_1:=\{c+1,\ldots,c+b\}$ and $I_2:=\{a+1,\ldots,b\}\cup\{d+1,\ldots,d+a\}$. Since $d+1>b$, the intervals defining $I_2$ are disjoint. Writing the two Gaussian polynomials with the common denominator $(q;q)_b$ gives
\[
\qbinom{b+c}{b}_q-\qbinom{a+d}{a}_q
=\frac{\prod_{i=1}^b(1-q^{c+i})-\prod_{i=a+1}^b(1-q^i)\prod_{i=1}^a(1-q^{d+i})}{(q;q)_b}.
\]
Here the first product in the numerator is $1+O(q^{a+2})$, whereas the second is $1-q^{a+1}+O(q^{a+2})$. Thus the numerator is $q^{a+1}$ times a polynomial with constant term $1$. At $q=-z$, the numerator becomes $\Pi_{I_1}(z)-\Pi_{I_2}(z)$. Hence
\[
H(z):=(-1)^{a+1}z^{-(a+1)}\bigl(\Pi_{I_1}(z)-\Pi_{I_2}(z)\bigr)
\]
is a polynomial with $H(0)=1$.
Let $F[X]:=(-1)^{a+1}(L_{b,c}[X]-L_{a,d}[X])$. Lemma~\ref{lem:binet} now gives
\begin{equation}\label{eq:H}
F\Bigl(1,\frac{z}{(1-z)^2}\Bigr)=\frac{z^{a+1}H(z)}{(1-z)^{ad}(-z;-z)_b}.
\end{equation}
Thus $F[X]$ has the form \eqref{eq:criterion-hyp} with $N=ad$, $m=a+1$ and $J(z)=H(z)$. The assumptions $a<b\le c$ give $ad=bc\ge(a+1)^2\ge2(a+1)$, so $a+1\le M$.

Let $B_i(z):=\Pi_{I_i}(z)-1$ for $i=1,2$. Since the constant terms cancel, $H(z)=(-1)^{a+1}z^{-(a+1)}(B_1(z)-B_2(z))$. Taking the negative part and evaluating at $z=1/2$ gives
\[
H^-(1/2)\le 2^{a+1}
\begin{cases}
B_1^-(1/2)+B_2^+(1/2),&a\text{ odd},\\
B_1^+(1/2)+B_2^-(1/2),&a\text{ even}.
\end{cases}
\]
Here the factor $2^{a+1}$ comes from division by $z^{a+1}$, and $B_i^-(1/2)=(\Pi_{I_i})^-(1/2)$. Since $w_{\mathrm{odd}}(I_i)\le2/3$ and $w_{\mathrm{even}}(I_i)\le1/3$, Lemma~\ref{lem:subset} applies to each term on the right-hand side.

\begin{proposition}\label{prop:equal-mass}
If $a\ge2$, then $H^-(1/2)<1$.
\end{proposition}

\begin{proof}
Let $x=2^{-(a+1)}$, so $x\le1/8$. Write $[k,\infty)$ for the integers at least $k$. We first find supersets of $I_1$ and $I_2$:
\begin{equation}\label{eq:supersets}
\begin{array}{c|c|c}
\text{case}&\text{superset of }I_1&\text{superset of }I_2\\\hline
b\ge a+3&[a+4,\infty)&[a+1,\infty)\\
b=a+1&[a+3,\infty)&\{a+1\}\cup[a+5,\infty)\\
b=a+2,\ a\ge4&[a+3,\infty)&\{a+1,a+2\}\cup[a+6,\infty)\\
(a,b)=(2,4)&[5,\infty)&\{3,4\}\cup[9,\infty)\\
(a,b)=(3,5)&[7,\infty)&\{4,5\}\cup[11,\infty)
\end{array}
\end{equation}
For $b\ge a+3$, the inclusions follow from $c+1\ge b+1\ge a+4$. If $b=a+1$, the equality $ad=(a+1)c$ implies $a\mid c$. Since $c\ge a+1$, we obtain $c\ge2a$ and $d\ge2a+2$. If $b=a+2$ and $a\ge4$,
\[
d=\frac{(a+2)c}{a}\ge\frac{(a+2)^2}{a}=a+4+\frac{4}{a}>a+4,
\]
so $d+1\ge a+6$. The remaining cases are $(a,b)=(2,4)$, where $d=2c\ge8$, and $(a,b)=(3,5)$, where $3d=5c$ forces $3\mid c$, so $c\ge6$ and $d\ge10$.

We bound the odd and even weighted sums for $I_1$ and $I_2$ by geometric series over the corresponding sets in \eqref{eq:supersets}. For example, suppose $b\ge a+3$ and $a$ is even. The odd integers in $[a+4,\infty)$ begin at $a+5$, while the even integers begin at $a+4$. Therefore
\[
w_{\mathrm{odd}}(I_1)\le\sum_{k\ge0}2^{-(a+5+2k)}=\frac{x}{12},\qquad
w_{\mathrm{even}}(I_1)\le\sum_{k\ge0}2^{-(a+4+2k)}=\frac{x}{6}.
\]
Similarly, $w_{\mathrm{odd}}(I_2)\le4x/3$ and $w_{\mathrm{even}}(I_2)\le2x/3$. Table~\ref{tab:masses} gives the bounds divided by $x$ for every case. Both bounds in Lemma~\ref{lem:subset} are power series in $w_{\mathrm{odd}}(I)$ and $w_{\mathrm{even}}(I)$ with nonnegative coefficients, so they increase in each of these quantities, and we may insert the upper bounds. With $2^{a+1}=x^{-1}$, this gives the following upper bound in the first row:
\[
H^-(1/2)\le
\frac{1/12+x/36}{(1-x/12)(1-x/6)}
+\frac{2/3}{(1-4x/3)(1-2x/3)}.
\]
The same substitution in each row gives a rational upper bound with nonnegative power-series coefficients in $x$. It is therefore increasing for $0<x\le x_*$, where $x_*$ is the largest admissible value of $x$ in that case. The last column of Table~\ref{tab:masses} lists these bounds at $x_*$. Every value is less than $63/64$, which proves the proposition.
\end{proof}

For $a=1$, we use the following identity of O'Hara \cite{OH} and Zeilberger \cite{Zei}, in the form stated by Zanello \cite[Lemma~4]{Zan}. Let $b,c\ge1$, and set $\lambda_j=0$ for $j>\ell(\lambda)$. Then
\begin{equation}\label{eq:KOH}
\qbinom{b+c}{b}_q
=\sum_{\lambda\vdash b}q^{\sum_i\lambda_i(\lambda_i-1)}
\prod_{j:\,\lambda_j>\lambda_{j+1}}
\qbinom{j(c+2)-\sum_{i<j}\lambda_i-\sum_{i\le j+1}\lambda_i}{\lambda_j-\lambda_{j+1}}_q.
\end{equation}
A factor whose top entry is smaller than its bottom entry is zero; we omit any summand containing such a factor below.
Using \eqref{eq:KOH} and the involution $\tau$ from \eqref{eq:tau}, we obtain an $e$-positive expansion of $L_{b,c}[X]$ in which one summand is $L_{1,bc}[X]$.

\begin{lemma}\label{lem:KOH}
For positive integers $b$, $c$ and every alphabet $X$, $L_{b,c}[X]-L_{1,bc}[X]\in\N[e_1[X],e_2[X]]$.
\end{lemma}

\begin{proof}
The top entry minus the bottom entry of the $j$th Gaussian factor in \eqref{eq:KOH} is $j(c+2)-2\sum_{i\le j}\lambda_i$. Moreover,
\[
\sum_jj(\lambda_j-\lambda_{j+1})=b,\qquad
\sum_j(\lambda_j-\lambda_{j+1})\sum_{i\le j}\lambda_i=\sum_j\lambda_j^2,
\]
so $\sum_j(\lambda_j-\lambda_{j+1})(j(c+2)-2\sum_{i\le j}\lambda_i)=bc-2\sum_i\lambda_i(\lambda_i-1)$. Replace $q^{\sum_i\lambda_i(\lambda_i-1)}$ by $(qt)^{\sum_i\lambda_i(\lambda_i-1)}$ and each Gaussian polynomial in \eqref{eq:KOH} by its two-variable homogenization. Every term then becomes homogeneous of degree $bc$ and reduces to the original term at $t=1$. Since $\qbinom{b+c}{b}_{q,t}$ is also homogeneous of degree $bc$ and reduces to $\qbinom{b+c}{b}_q$ at $t=1$, we obtain
\[
\qbinom{b+c}{b}_{q,t}
=\sum_{\lambda\vdash b}(qt)^{\sum_i\lambda_i(\lambda_i-1)}
\prod_{j:\,\lambda_j>\lambda_{j+1}}
\qbinom{j(c+2)-\sum_{i<j}\lambda_i-\sum_{i\le j+1}\lambda_i}{\lambda_j-\lambda_{j+1}}_{q,t}.
\]
By \eqref{eq:tau}, the left side equals $\tau(L_{b,c}[q+t])$, and each factor $\qbinom{n}{k}_{q,t}$ on the right equals $\tau(L_{k,n-k}[q+t])$. Specialization to $q+t$ is injective on $\Z[e_1[X],e_2[X]]$, and the exponent $\sum_i\lambda_i(\lambda_i-1)$ is even. Hence applying $\tau$ gives
\[
L_{b,c}[X]=\sum_{\lambda\vdash b}e_2[X]^{\sum_i\lambda_i(\lambda_i-1)}
\prod_{j:\,\lambda_j>\lambda_{j+1}}
L_{\lambda_j-\lambda_{j+1},\,j(c+2)-2\sum_{i\le j}\lambda_i}[X].
\]
Every summand belongs to $\N[e_1[X],e_2[X]]$. For $\lambda=(1^b)$, only the factor with $j=b$ occurs; its bottom entry is $1$, and its top entry is $bc+1$. The corresponding summand is $L_{1,bc}[X]$. Subtracting it proves the lemma.

\begin{table}[tb]
\centering
\caption{Bounds for the weighted sums and $H^-(1/2)$}
\label{tab:masses}
\small
\renewcommand{\arraystretch}{1.2}
\setlength{\tabcolsep}{7pt}
\begin{tabular}{@{}lcccc@{}}
\toprule
Case & $I_1$ & $I_2$ & $x_*$ & $H^-(1/2)$ \\
\midrule
$b\ge a+3$, $a$ even
  & $(1/12,1/6)$ & $(4/3,2/3)$
  & $1/8$ & $<\frac{31}{32}$ \\
$b\ge a+3$, $a$ odd
  & $(1/6,1/12)$ & $(2/3,4/3)$
  & $1/16$ & $<\frac{63}{64}$ \\
\addlinespace[3pt]
$b=a+1$, $a$ even
  & $(1/3,1/6)$ & $(13/12,1/24)$
  & $1/8$ & $<\frac12$ \\
$b=a+1$, $a$ odd
  & $(1/6,1/3)$ & $(1/24,13/12)$
  & $1/16$ & $<\frac12$ \\
\addlinespace[3pt]
$b=a+2$, $a\ge4$ even
  & $(1/3,1/6)$ & $(49/48,13/24)$
  & $1/32$ & $<\frac{15}{16}$ \\
$b=a+2$, $a\ge5$ odd
  & $(1/6,1/3)$ & $(13/24,49/48)$
  & $1/64$ & $<\frac{15}{16}$ \\
\addlinespace[3pt]
$(a,b)=(2,4)$
  & $(1/3,1/6)$ & $(49/48,49/96)$
  & $1/8$ & $<\frac{63}{64}$ \\
$(a,b)=(3,5)$
  & $(1/6,1/12)$ & $(49/96,193/192)$
  & $1/16$ & $<\frac34$ \\
\bottomrule
\end{tabular}
\end{table}
\end{proof}

We are now in a position to prove Theorem~\ref{thm:equal}.

\begin{proof}[Proof of Theorem~\ref{thm:equal}]
If $a=b$, then $c=d$ and the difference is zero. Suppose $a<b$, and let $F[X]=(-1)^{a+1}(L_{b,c}[X]-L_{a,d}[X])$ and $C_j:=c_j(F)$. Equation~\eqref{eq:H} and Proposition~\ref{prop:kernel} give $C_j=0$ for $j<a+1$ and $C_{a+1}=H(0)=1$.

If $a\ge2$, Proposition~\ref{prop:equal-mass} gives $H^-(1/2)<1$, so Corollary~\ref{cor:criterion} yields $C_j\ge\binom{M-a-1}{j-a-1}$. If $a=1$, Lemma~\ref{lem:KOH} gives $F[X]=L_{b,c}[X]-L_{1,bc}[X]\in\N[e_1[X],e_2[X]]$. Write $F(1,y)=\sum_k\gamma_ky^k$, with $\gamma_k\ge0$. The relations $C_0=C_1=0$ and $C_2=1$, together with \eqref{eq:transform}, imply $\gamma_0=\gamma_1=0$ and $\gamma_2=1$. Applying \eqref{eq:transform} again, we obtain
\[
C_j=\binom{M-2}{j-2}+\sum_{k=3}^j\binom{M-k}{j-k}\gamma_k\ge\binom{M-2}{j-2}.
\]
In both cases, the lower bound is equivalent to
\[
(-1)^{a+1}(L_{b,c}[X]-L_{a,d}[X])-e_1[X]^{ad-2(a+1)}e_2[X]^{a+1}\in\N[e_1[X],e_2[X],h_2[X]].
\]
The Pieri rule now gives the Schur positivity asserted in Theorem~\ref{thm:equal}.
\end{proof}

\begin{example}\label{ex:equal}
Take $(a,b,c,d)=(2,3,4,6)$, so that $ad=12$ and $M=6$. From \eqref{eq:Lrs} one computes
\[
\begin{aligned}
L_{2,6}[X]-L_{3,4}[X]
&=e_1[X]^6e_2[X]^3+5e_1[X]^4e_2[X]^4+6e_1[X]^2e_2[X]^5+e_2[X]^6\\
&=h_2[X]^3e_2[X]^3+8h_2[X]^2e_2[X]^4+19h_2[X]e_2[X]^5+13e_2[X]^6.
\end{aligned}
\]
Here $(-1)^{a+1}=-1$ and $(C_3,C_4,C_5,C_6)=(1,8,19,13)$, so the bounds $C_j\ge\binom{3}{j-3}$ of Theorem~\ref{thm:equal} hold. In two variables, the Schur expansion is
\[
L_{2,6}[q+t]-L_{3,4}[q+t]=s_{(9,3)}[q+t]+10s_{(8,4)}[q+t]+30s_{(7,5)}[q+t]+22s_{(6,6)}[q+t].
\]
\end{example}

\begin{proof}[Proof of Corollary~\ref{cor:bergeron}]
Suppose first that $c>a$. Then $d/b=c/a>1$ and $a\le\min(b,c)\le\max(b,c)\le d$. Theorem~\ref{thm:equal}, applied to $(a,\min(b,c),\max(b,c),d)$ together with $L_{b,c}[X]=L_{c,b}[X]$, gives the required sign. Since $C_{a+1}=1$ there, the difference vanishes exactly when $a=\min(b,c)$, that is, when $a=b$. If $c<a$, apply the first case to $(c,d,a,b)$ in place of $(a,b,c,d)$. This changes the signs of both the difference and $\varepsilon$, and the difference vanishes exactly when $c=d$, that is, when $a=b$. If $a=c$, then $b=d$ and the difference vanishes. Hence the difference vanishes exactly when $a=b$ or $a=c$.

A nonzero element of $\N[e_1[X],e_2[X],h_2[X]]$ has a nonzero specialization to $q+t$, because the three generators specialize to nonzero polynomials with nonnegative coefficients. Hence it remains nonzero for every alphabet with at least two letters. A nonzero Schur-positive polynomial cannot have a Schur-positive negative, so the sign is unique.
\end{proof}

Theorem~\ref{thm:equal} also determines the first nonzero coefficient in the two-variable Schur expansion. Setting $q=1$, so that $e_1[1+t]=1+t$ and $e_2[1+t]=t$, the right side of \eqref{eq:equal} is $t^{a+1}+O(t^{a+2})$. On the other hand, \eqref{eq:two-row} gives $s_{(ad-j,j)}[1+t]=t^j(1+t+\cdots+t^{ad-2j})$. Hence the coefficient of $s_{(ad-j,j)}[q+t]$ in the signed difference is zero for $j<a+1$ and one for $j=a+1$.

\subsection{Proof of Theorem~\ref{thm:unequal}}\label{sec:unequal}

Since $L_{1,n}[X]=U_{n+1}[X]$, Lemma~\ref{lem:KOH} gives $L_{u,v}[X]-U_{uv+1}[X]\in\N[e_1[X],e_2[X]]$. We next compare $U_{rs+2}[X]-e_1[X]L_{r,s}[X]$.

Fix $2\le r\le s$. The products introduced in Subsection~\ref{sec:criterion} give $(-z;-z)_r=(1+z)\Pi_{\{2,\ldots,r\}}(z)$ and $\qbinom{r+s}{r}_{-z}=\Pi_{\{s+1,\ldots,s+r\}}(z)/(-z;-z)_r$. Define
\begin{equation}\label{eq:Jrs}
J_{r,s}(z):=\frac{\bigl(1-(-z)^{rs+2}\bigr)\Pi_{\{2,\ldots,r\}}(z)-(1-z)\Pi_{\{s+1,\ldots,s+r\}}(z)}{z}.
\end{equation}
The two $\Pi$-products in \eqref{eq:Jrs} equal $1+O(z^2)$ and $1+O(z^3)$, respectively. Thus the numerator is $z+O(z^2)$, so $J_{r,s}(z)\in\Z[z]$ and $J_{r,s}(0)=1$. Let $F[X]:=U_{rs+2}[X]-e_1[X]L_{r,s}[X]$, which is homogeneous of degree $rs+1$ in $e_1[X]$ and $e_2[X]$. Lemma~\ref{lem:binet} gives
\[
F\Bigl(1,\frac{z}{(1-z)^2}\Bigr)=\frac{[rs+2]_{-z}-(1-z)\qbinom{r+s}{r}_{-z}}{(1-z)^{rs+1}},
\]
and
\[
[rs+2]_{-z}=\frac{1-(-z)^{rs+2}}{1+z}=\frac{\bigl(1-(-z)^{rs+2}\bigr)\Pi_{\{2,\ldots,r\}}(z)}{(-z;-z)_r}.
\]
Hence
\begin{equation}\label{eq:Psi}
F\Bigl(1,\frac{z}{(1-z)^2}\Bigr)=\frac{z\,J_{r,s}(z)}{(1-z)^{rs+1}(-z;-z)_r}.
\end{equation}

We use $(f(z)+g(z))^-(1/2)\le f^-(1/2)+g^-(1/2)$ and, when $f(0)=0$, $(f(z)/z)^-(1/2)=2f^-(1/2)$.

\begin{lemma}\label{lem:J}
For every $2\le r\le s$, we have $J_{r,s}^-(1/2)<1$.
\end{lemma}

\begin{proof}
Let $P(z):=\Pi_{\{2,\ldots,r\}}(z)$ and
$Q(z):=\Pi_{\{s+1,\ldots,s+r\}}(z)$.
Rearranging \eqref{eq:Jrs} gives
\[
J_{r,s}(z)
=1+\frac{P(z)-1}{z}-\frac{Q(z)-1}{z}
 +(Q(z)-1)-\frac{(-z)^{rs+2}P(z)}{z}.
\]
By subadditivity of the negative part, and since each coefficient of $P(z)$ is at most the corresponding coefficient of $\prod_{i=2}^r(1+z^i)$ in absolute value,
\[
\begin{aligned}
J_{r,s}^-(1/2)
\le 2(P-1)^-(1/2)+2(Q-1)^+(1/2)+(Q-1)^-(1/2)+2^{-(rs+1)}\prod_{i=2}^r(1+2^{-i}).
\end{aligned}
\]

Suppose first that $s\ge3$, and set
$\Delta:=(1-1/24)(1-1/144)$.
For any finite set $I\subseteq\{4,5,\ldots\}$, write
$x=w_{\mathrm{odd}}(I)\le1/24$ and
$y=w_{\mathrm{even}}(I)\le1/12$.
As in the proof of Lemma~\ref{lem:subset}, the total weight of subsets containing $h$ odd indices and $k$ even indices is at most $x^hy^k$. Such a subset has sign $(-1)^k$ in \eqref{eq:subsets}. Summing separately over even and odd $k$, and noting that cancellation can only decrease the positive and negative contributions, gives
\[
\begin{aligned}
(\Pi_I)^+(1/2)\le\frac1{(1-x)(1-y^2)}\le\frac1\Delta,\qquad
(\Pi_I)^-(1/2)\le\frac{y}{(1-x)(1-y^2)}\le\frac1{12\Delta}.
\end{aligned}
\]
Since the empty subset contributes the constant term $1$,
\[
(\Pi_I-1)^+(1/2)=(\Pi_I)^+(1/2)-1
\le\frac1\Delta-1\le\frac7{144\Delta}.
\]
Every element of $\{s+1,\ldots,s+r\}$ is at least $4$, so these bounds give
\[
(Q-1)^+(1/2)\le\frac7{144\Delta},\qquad
(Q-1)^-(1/2)\le\frac1{12\Delta}.
\]

For $r\ge3$, write $P(z)=A(z)R(z)$, where
\[
A(z):=(1-z^2)(1+z^3),\qquad
R(z):=\Pi_{\{4,\ldots,r\}}(z).
\]
The same bounds give $R^+(1/2)\le1/\Delta$ and $R^-(1/2)\le1/(12\Delta)$.
Since $A(z)=1-z^2+z^3-z^5$, we have $A^+(1/2)=9/8$ and $A^-(1/2)=9/32$. Consequently,
\[
\begin{aligned}
2(P-1)^-(1/2)
&\le 2\bigl(A^+(1/2)R^-(1/2)+A^-(1/2)R^+(1/2)\bigr)\\
&\le 2\left(\frac9{8}\frac1{12\Delta}+\frac9{32}\frac1\Delta\right)
=\frac3{4\Delta}.
\end{aligned}
\]
This bound also holds for $r=2$, since $P(z)=1-z^2$ and $2(P-1)^-(1/2)=1/2$.

Finally, $\prod_{i=2}^r(1+2^{-i})<\exp\bigl(\sum_{i\ge2}2^{-i}\bigr)=e^{1/2}<2$. Since $rs\ge6$, the last term in the initial estimate is less than $2^{-rs}\le1/64$. Combining the estimates, we obtain
\[
J_{r,s}^-(1/2)
<\frac1\Delta\left(\frac34+\frac7{72}+\frac1{12}\right)+\frac1{64}
=\frac{67}{72\Delta}+\frac1{64}<1.
\]

For $r=s=2$, direct expansion gives
$J_{2,2}(z)=1-z-z^2+2z^3-z^4-z^5+z^6$,
so $J_{2,2}^-(1/2)=27/32<1$.
\end{proof}

The estimate in Lemma~\ref{lem:J} allows us to apply Corollary~\ref{cor:criterion} to \eqref{eq:Psi}.

\begin{proposition}\label{prop:upper}
For positive integers $r,s$ and every alphabet $X$,
\[
U_{rs+2}[X]-e_1[X]L_{r,s}[X]-e_1[X]^{rs-1}e_2[X]\in\N[e_1[X],e_2[X],h_2[X]].
\]
\end{proposition}

\begin{proof}
By symmetry, assume $r\le s$. If $r=1$, the recurrence gives $U_{s+2}[X]-e_1[X]L_{1,s}[X]=e_2[X]U_s[X]$. Equation~\eqref{eq:U-explicit} gives $U_s[X]-e_1[X]^{s-1}\in\N[e_1[X],e_2[X]]$, which proves the assertion.

If $r\ge2$, Equation~\eqref{eq:Psi} and Lemma~\ref{lem:J} verify the hypotheses of Corollary~\ref{cor:criterion} for $U_{rs+2}[X]-e_1[X]L_{r,s}[X]$ with $N=rs+1$, $m=1$, $b=r$, and $J(z)=J_{r,s}(z)$. Thus $U_{rs+2}[X]-e_1[X]L_{r,s}[X]-e_1[X]^{rs-1}e_2[X]\in\N[e_1[X],e_2[X],h_2[X]]$.
\end{proof}

These comparisons provide the ingredients for Theorem~\ref{thm:unequal}.

\begin{proof}[Proof of Theorem~\ref{thm:unequal}]
Let $F[X]:=L_{u,v}[X]-e_1[X]^{uv-rs}L_{r,s}[X]$ and $D_j:=c_j(F)$. Since $uv-rs\ge1$, write
\begin{equation}\label{eq:decomp}
\begin{split}
L_{u,v}[X]-e_1[X]^{uv-rs}L_{r,s}[X]
&=(L_{u,v}[X]-U_{uv+1}[X])+(U_{uv+1}[X]-e_1[X]^{uv-rs-1}U_{rs+2}[X])\\
&\qquad+e_1[X]^{uv-rs-1}(U_{rs+2}[X]-e_1[X]L_{r,s}[X]).
\end{split}
\end{equation}
The first term lies in $\N[e_1[X],e_2[X]]$ by Lemma~\ref{lem:KOH}. For the middle term, applying the recurrence \eqref{eq:U} $uv-rs-1$ times gives
\[
U_{uv+1}[X]-e_1[X]^{uv-rs-1}U_{rs+2}[X]
=e_2[X]\sum_{i=0}^{uv-rs-2}e_1[X]^iU_{uv-1-i}[X].
\]
The sum is empty when $uv-rs=1$. For each $i$, the full summand $e_2[X]e_1[X]^iU_{uv-1-i}[X]$ has $e_1[X]^{uv-2}e_2[X]$ as its term with the largest power of $e_1[X]$, by \eqref{eq:U-explicit}; subtracting this term leaves an element of $\N[e_1[X],e_2[X]]$. Thus the middle term in \eqref{eq:decomp}, less $(uv-rs-1)e_1[X]^{uv-2}e_2[X]$, is $e$-positive. Proposition~\ref{prop:upper} shows that the last term, less $e_1[X]^{uv-2}e_2[X]$, lies in $\N[e_1[X],e_2[X],h_2[X]]$. Adding these expressions proves \eqref{eq:unequal-base}.

By \eqref{eq:transform}, the coefficient of $e_1[X]^\delta h_2[X]^{M-j}e_2[X]^j$ in $e_1[X]^{uv-2}e_2[X]$ is $\binom{M-1}{j-1}$. Lemma~\ref{lem:basis}(b) therefore gives $D_j\ge(uv-rs)\binom{M-1}{j-1}$. Finally, \eqref{eq:Lrs-lin} gives constant term zero and linear coefficient $(uv-1)-(rs-1)=uv-rs$ for $L_{u,v}(1,y)-L_{r,s}(1,y)$. Applying \eqref{eq:transform} yields $D_0=0$ and $D_1=uv-rs$.
\end{proof}

\section{Lucas multinomial differences}\label{sec:multi}

In Subsection~\ref{sec:dominance} we construct a chain between two partitions in dominance order and compute the change in degree at each move. In Subsection~\ref{sec:adjacent} we use \eqref{eq:atoms} to prove $e$-positivity for one move. Subsection~\ref{sec:multi-proof} combines these comparisons to prove Theorem~\ref{thm:multi}.

\subsection{Degrees and dominance chains}\label{sec:dominance}

We also use $L_\alpha[X]$ when $\alpha$ is a list of nonnegative integers with positive sum. The factorial quotient is unchanged by rearranging the entries or omitting zeros, since $U_0![X]=1$. In particular, $L_{(n)}[X]=1$.

For a partition $\alpha$, let $\rho(\alpha):=[e_1[X]^{d(\alpha)-2}e_2[X]]L_\alpha[X]$ when $d(\alpha)\ge2$, and let $\rho(\alpha):=0$ otherwise.

\begin{lemma}\label{lem:factor}
Let $\alpha=(\alpha_1,\ldots,\alpha_\ell)\vdash n\ge1$. Then
\begin{equation}\label{eq:factor}
L_\alpha[X]=\prod_{i=2}^\ell L_{\alpha_1+\cdots+\alpha_{i-1},\alpha_i}[X]\in\N[e_1[X],e_2[X]].
\end{equation}
In particular, $L_\alpha[X]$ is homogeneous of degree $d(\alpha)=\sum_{i<j}\alpha_i\alpha_j=(n^2-\sum_i\alpha_i^2)/2$, and $L_\alpha[X]-e_1[X]^{d(\alpha)}\in\N[e_1[X],e_2[X]]$. Moreover, $L_\alpha(1,0)=1$ and
\begin{equation}\label{eq:rho}
[y]\,L_\alpha(1,y)=d(\alpha)-\ell(\alpha)+1,
\end{equation}
so that $\rho(\alpha)=d(\alpha)-\ell(\alpha)+1$.
\end{lemma}

\begin{proof}
Cancelling factorials in the product on the right of \eqref{eq:factor} gives $L_\alpha[X]$. Its factors are $e$-positive and have total degree $\sum_{i<j}\alpha_i\alpha_j=d(\alpha)$. After the coordinate substitution $(e_1[X],e_2[X])=(1,y)$, each factor has constant term $1$. Hence $L_\alpha(1,0)=1$, and the coefficient of $e_1[X]^{d(\alpha)}$ in $L_\alpha[X]$ is $1$. This proves $L_\alpha[X]-e_1[X]^{d(\alpha)}\in\N[e_1[X],e_2[X]]$.

For \eqref{eq:rho}, use the factorial quotient defining $L_\alpha[X]$. The recurrence gives $U_1(1,y)=1$ and $U_i(1,y)=1+(i-2)y+O(y^2)$ for $i\ge2$. Comparing linear coefficients in the quotient yields
\[
[y]\,L_\alpha(1,y)=\frac{(n-1)(n-2)}{2}-\sum_{t=1}^\ell\frac{(\alpha_t-1)(\alpha_t-2)}{2}=\frac{n^2-\sum_t\alpha_t^2}{2}-\ell+1=d(\alpha)-\ell+1.\qedhere
\]
\end{proof}

A dominance relation $\alpha\unrhd\beta$ can be realized by moving cells to lower rows of a Young diagram \cite{Bry}. The following lemma gives a specific sequence of such moves.

\begin{lemma}\label{lem:balance}
Let $\alpha\dom\beta$ be partitions of $n$. There is a chain $\alpha=\alpha^{(0)},\alpha^{(1)},\ldots,\alpha^{(k)}=\beta$ of partitions, each dominating $\beta$, such that each step $\mu\to\mu'$ replaces two entries $A$ and $B$ with $A\ge B+2$ by $A-1$ and $B+1$; here $B=0$ is allowed, in which case the step creates a new part equal to $1$. For each step, $d(\mu')-d(\mu)=A-B-1\ge1$, while $\rho(\mu')-\rho(\mu)=A-B-1$ if $B\ge1$ and $A-2$ if $B=0$.
\end{lemma}

\begin{proof}
Append zeros so that every partition of $n$ becomes a sequence with $n$ entries. Suppose $\mu\dom\beta$ and $\mu\ne\beta$. Let $p$ be the first index with $\mu_p\ne\beta_p$. Dominance gives $\mu_p>\beta_p$. Since $|\mu|=|\beta|$, there is an index $q>p$ with $\mu_q<\beta_q$; choose the first such index. Let $i$ be the last index with $\mu_i=\mu_p$. Then
\[
\mu_p\ge\beta_p+1\ge\beta_q+1\ge\mu_q+2,
\]
so $p\le i<q$. Let $A:=\mu_i$ and $B:=\mu_q$. Decrease $\mu_i$ by $1$ and increase $\mu_q$ by $1$.

The resulting sequence is a partition. The choice of $i$ gives $\mu_{i+1}<\mu_i$. If $i<q-1$, then $q>p+1$ and $\mu_{q-1}\ge\beta_{q-1}\ge\beta_q>\mu_q$, by the choice of $q$. If $i=q-1$, the changed entries remain in order because $A-1\ge B+1$.

Only the partial sums ending at indices $i\le j<q$ decrease, each by one. Before the move, each of these sums exceeds the corresponding sum of $\beta$ by at least $\mu_p-\beta_p\ge1$. Thus the new partition still dominates $\beta$. Moreover, $\mu_i=\mu_p>\beta_p\ge\beta_i$ and $\mu_q<\beta_q$, so $\sum_j|\mu_j-\beta_j|$ decreases by two. Repeating the construction therefore reaches $\beta$.

The move decreases the sum of squared parts by $2(A-B-1)$, so the degree formula in Lemma~\ref{lem:factor} gives $d(\mu')-d(\mu)=A-B-1$. The length is unchanged for $B\ge1$ and increases by one for $B=0$. The stated changes in $\rho(\mu)$ follow from $\rho(\mu)=d(\mu)-\ell(\mu)+1$.
\end{proof}

\subsection{Comparison for consecutive partitions}\label{sec:adjacent}

For consecutive partitions in this chain, we use the Lucas atoms of Sagan and Tirrell \cite[Proposition~2.2]{ST}. These are polynomials $P_r[X]\in\N[e_1[X],e_2[X]]$, indexed by $r\ge2$, satisfying
\begin{equation}\label{eq:atoms}
U_n[X]=\prod_{r\mid n,\ r\ge2}P_r[X]\qquad(n\ge1).
\end{equation}
For example, $P_2[X]=e_1[X]$, $P_3[X]=e_1[X]^2+e_2[X]$ and $P_4[X]=e_1[X]^2+2e_2[X]$, with $U_4[X]=P_2[X]P_4[X]$.

\begin{lemma}\label{lem:atoms}
Let $n_1,\ldots,n_s$ and $m_1,\ldots,m_t$ be positive integers such that, for every $r\ge2$, at least as many of the $n_i$ as of the $m_j$ are divisible by $r$. Then
\[
\frac{U_{n_1}[X]U_{n_2}[X]\cdots U_{n_s}[X]}{U_{m_1}[X]U_{m_2}[X]\cdots U_{m_t}[X]}\in\N[e_1[X],e_2[X]].
\]
\end{lemma}

\begin{proof}
By \eqref{eq:atoms}, the exponent of $P_r[X]$ in the quotient is the number of $n_i$ divisible by $r$ minus the number of $m_j$ divisible by $r$. The hypothesis makes every exponent nonnegative. The quotient is therefore a product of $e$-positive Lucas atoms. This is a special case of Sagan and Tirrell \cite[Theorem~1.1]{ST}.
\end{proof}

\begin{lemma}\label{lem:local}
Let $\mu'$ be obtained from $\mu$ by one step of Lemma~\ref{lem:balance}, and let $t:=d(\mu')-d(\mu)$. Then
\begin{equation}\label{eq:local-positive}
L_{\mu'}[X]-e_1[X]^tL_\mu[X]\in\N[e_1[X],e_2[X]],
\end{equation}
and this difference is zero exactly when the step has $(A,B)=(2,0)$, that is, when it replaces a part equal to $2$ by two parts equal to $1$.
\end{lemma}

\begin{proof}
Suppose the move replaces $A,B$ by $A-1,B+1$, where $A\ge B+2$ and $B\ge0$. Then $t=A-B-1$. Let $\eta$ be the list of remaining parts, and let $G[X]:=L_{(A+B,\eta)}[X]$. Factoring the two multinomial coefficients gives
\[
L_\mu[X]=G[X]L_{B,A}[X],\qquad L_{\mu'}[X]=G[X]L_{B+1,A-1}[X].
\]
By Lemma~\ref{lem:factor}, applied after sorting the entries of $(A+B,\eta)$, $G[X]$ is nonzero and $e$-positive. It suffices to prove the assertion for $L_{B+1,A-1}[X]-e_1[X]^tL_{B,A}[X]$.

By \eqref{eq:Lrs}, $L_{B+1,A-1}[X]=L_{B,A}[X]U_A[X]/U_{B+1}[X]$ in the fraction field of $\Z[e_1[X],e_2[X]]$. Let $g:=\gcd(A,B+1)$. Then
\begin{equation}\label{eq:local-atom-factor}
L_{B+1,A-1}[X]-e_1[X]^tL_{B,A}[X]
=\frac{U_g[X]L_{B,A}[X]}{U_{B+1}[X]}\cdot\frac{U_A[X]-e_1[X]^tU_{B+1}[X]}{U_g[X]}.
\end{equation}
For $B=0$ this reads $U_A[X]-e_1[X]^{A-1}=1\cdot(U_A[X]-e_1[X]^{A-1})$.

Consider first $U_g[X]L_{B,A}[X]/U_{B+1}[X]=U_g[X]U_{A+B}![X]/(U_{B+1}[X]U_A![X]U_B![X])$. Fix $r\ge2$. The number of multiples of $r$ in $\{1,\ldots,m\}$ is $\lfloor m/r\rfloor$, and
\[
\Bigl\lfloor\frac{A+B}{r}\Bigr\rfloor\ge\Bigl\lfloor\frac{A}{r}\Bigr\rfloor+\Bigl\lfloor\frac{B}{r}\Bigr\rfloor.
\]

If $r\nmid B+1$, this inequality verifies the condition in Lemma~\ref{lem:atoms}. Suppose $r\mid B+1$, so that $B\equiv-1\pmod r$. If $r\nmid A$, the remainders of $A$ and $B$ modulo $r$ sum to at least $r$, giving $\lfloor(A+B)/r\rfloor=\lfloor A/r\rfloor+\lfloor B/r\rfloor+1$. If $r\mid A$, then $r\mid g$, and the extra factor $U_g[X]$ supplies the required multiple. Thus the condition holds for every $r$, and the first factor is a nonzero element of $\N[e_1[X],e_2[X]]$.

For the second factor, write $B+1=gu$ and $A=gv$, where $1\le u<v$. Then $t=g(v-u)$, and summing consecutive differences gives
\begin{equation}\label{eq:local-Lucas-sum}
\frac{U_A[X]-e_1[X]^tU_{B+1}[X]}{U_g[X]}
=\sum_{j=u}^{v-1}e_1[X]^{g(v-j-1)}\,\frac{U_{g(j+1)}[X]-e_1[X]^gU_{gj}[X]}{U_g[X]}.
\end{equation}
Thus it suffices to show that
\[
\frac{U_{g(j+1)}[X]-e_1[X]^gU_{gj}[X]}{U_g[X]}
\]
is $e$-positive for $u\le j<v$.
We use the identity
\begin{equation}\label{eq:addition}
U_{m+n}[X]=U_m[X]U_{n+1}[X]+e_2[X]U_{m-1}[X]U_n[X]\qquad(m\ge1,\ n\ge0).
\end{equation}
It holds for $n=0$ and $n=1$ by \eqref{eq:U}, and if it holds for $n-1$ and $n$, then
\begin{align*}
U_{m+n+1}[X]&=e_1[X]U_{m+n}[X]+e_2[X]U_{m+n-1}[X]\\
&=U_m[X]\bigl(e_1[X]U_{n+1}[X]+e_2[X]U_n[X]\bigr)\\
&\quad+e_2[X]U_{m-1}[X]\bigl(e_1[X]U_n[X]+e_2[X]U_{n-1}[X]\bigr)\\
&=U_m[X]U_{n+2}[X]+e_2[X]U_{m-1}[X]U_{n+1}[X].
\end{align*}
With $m=gj$ and $n=g$, \eqref{eq:addition} gives
\begin{equation}\label{eq:local-Lucas-increment}
\frac{U_{g(j+1)}[X]-e_1[X]^gU_{gj}[X]}{U_g[X]}
=\frac{U_{gj}[X]}{U_g[X]}\,\bigl(U_{g+1}[X]-e_1[X]^g\bigr)+e_2[X]U_{gj-1}[X].
\end{equation}
Since every divisor of $g$ divides $gj$, Lemma~\ref{lem:atoms} gives $U_{gj}[X]/U_g[X]\in\N[e_1[X],e_2[X]]$. It follows from \eqref{eq:U-explicit} that
\[
U_{g+1}[X]-e_1[X]^g
=\sum_{k=1}^{\lfloor g/2\rfloor}
\binom{g-k}{k}e_1[X]^{g-2k}e_2[X]^k
\in\N[e_1[X],e_2[X]].
\]
Hence each term in \eqref{eq:local-Lucas-increment} is $e$-positive. Equation~\eqref{eq:local-Lucas-sum} shows that the second factor in \eqref{eq:local-atom-factor} is $e$-positive as well. This proves \eqref{eq:local-positive}.

It remains to determine the equality case. The first factor in \eqref{eq:local-atom-factor} is nonzero, and all terms in \eqref{eq:local-Lucas-sum} are $e$-positive. The difference therefore vanishes if and only if every term in this sum is zero. For $g\ge2$, the polynomial $U_{g+1}[X]-e_1[X]^g$ is nonzero, so every term is nonzero. For $g=1$, the $j$th term is $e_1[X]^{v-j-1}e_2[X]U_{j-1}[X]$, which vanishes only for $j=1$. Thus equality holds precisely when $g=1$, $u=1$ and $v=2$, or equivalently $(A,B)=(2,0)$.
\end{proof}

\subsection{Proof of Theorem~\ref{thm:multi}}\label{sec:multi-proof}

We can now prove Theorem~\ref{thm:multi} by combining the comparisons for consecutive partitions along a dominance chain.

\begin{proof}[Proof of Theorem~\ref{thm:multi}]
Take a chain $\alpha=\alpha^{(0)},\ldots,\alpha^{(k)}=\beta$ from Lemma~\ref{lem:balance}, allowing $k=0$ when $\alpha=\beta$. Let $d_i:=d(\alpha^{(i)})$, $t_i:=d_{i+1}-d_i$ and $\kappa_i:=\rho(\alpha^{(i+1)})-\rho(\alpha^{(i)})$. Lemma~\ref{lem:balance} gives $t_i>0$ and $\kappa_i\ge0$. By Lemma~\ref{lem:factor}, $\sum_i\kappa_i=\rho(\beta)-\rho(\alpha)=d(\beta)-d(\alpha)-\ell(\beta)+\ell(\alpha)=\kappa$. Hence $d(\beta)-d(\alpha)=\sum_it_i\ge0$ and $\kappa\ge0$. Summing the adjacent differences gives
\[
L_\beta[X]-e_1[X]^{d(\beta)-d(\alpha)}L_\alpha[X]
=\sum_{i=0}^{k-1}e_1[X]^{d(\beta)-d_{i+1}}\bigl(L_{\alpha^{(i+1)}}[X]-e_1[X]^{t_i}L_{\alpha^{(i)}}[X]\bigr).
\]
Every summand is $e$-positive by Lemma~\ref{lem:local}, proving \eqref{eq:multi-positive}.

After the coordinate substitution $(e_1[X],e_2[X])=(1,y)$, Lemma~\ref{lem:factor} gives constant term zero and linear coefficient $\kappa$ for the difference. If $\kappa>0$, then $d(\beta)\ge2$, and homogeneity shows that the coefficient of $e_1[X]^{d(\beta)-2}e_2[X]$ is $\kappa$. Subtracting this term leaves an $e$-positive polynomial, which proves \eqref{eq:multi-base}. In particular, the difference is nonzero.

If $\kappa=0$, every $\kappa_i$ is zero. Lemma~\ref{lem:balance} then forces each step to replace $(2,0)$ by $(1,1)$. Consequently, $\beta$ is obtained from $\alpha$ by splitting some parts equal to $2$ into two parts equal to $1$. Conversely, each such replacement multiplies $L_\alpha[X]$ by $U_2[X]=e_1[X]$ and increases its degree by one. A sequence of these replacements gives $L_\beta[X]=e_1[X]^{d(\beta)-d(\alpha)}L_\alpha[X]$, and the preceding paragraph implies $\kappa=0$.
\end{proof}

\par\medskip\noindent\textbf{AI declaration.} During the preparation of this work, the authors used GPT-5.6 Sol (\mbox{OpenAI}) to assist with proof development, language editing, proofreading, and drafting parts of the manuscript. The authors reviewed and revised this content, independently verified the proofs, and take full responsibility for the manuscript.

\end{document}